\documentclass{article}
\usepackage[english]{babel}
\usepackage{amsmath,amssymb,amsthm}
\usepackage{amsbib}
\usepackage{youngtab}
\usepackage{ytableau}
\usepackage{comment}

\theoremstyle{theorem}
\newtheorem{theorem}{Theorem}[section]

\newtheorem{proposition}{Proposition}[theorem]

\theoremstyle{definition}
\newtheorem{definition}[theorem]{Definition}
\newtheorem{example}[theorem]{Example}
\theoremstyle{remark}
\newtheorem{remark}[theorem]{Remark}

\title{
    Series of quasi-uniform scatterings with fast search,
    root systems and neural network classifications
}
\author{
  Igor V. Netay
  \thanks{Joint Stock "Research and Production company ``Kryptonite"}
  \thanks{Institute for Information Transmission Problems, Russian Academy of Sciences}
  \href{mailto:i.netay@kryptonite.ru}{i.netay@kryptonite.ru}
}
\date{}
\begin{document}
\maketitle

\begin{abstract}
    In this paper we describe an approach to construct large extendable 
    collections of vectors in predefined spaces of given dimensions. 
    These collections are useful for neural network latent space configuration 
    and training.
    For classification problem with large or unknown number of classes this 
    allows to construct classifiers without classification layer 
    and extend the number of classes without retraining of network from the very beginning.
    The construction allows to create large well-spaced vector collections 
    in spaces of minimal possible dimension. 
    If the number of classes is known or approximately predictable,
    one can choose sufficient enough vector collection size. 
    If one needs to significantly extend the number of classes, one can 
    extend the collection in the same latent space, or to incorporate 
    the collection into collection of higher dimensions with same 
    spacing between vectors. 
    Also, regular symmetric structure of constructed vector collections 
    can significantly simplify problems of search 
    for nearest cluster centers or embeddings in the latent space.
    Construction of vector collections is based on combinatorics 
    and geometry of semi-simple Lie groups irreducible representations with highest weight.
\end{abstract}

\section{Introduction}

Classification problems with large number of classes arise frequently 
in modern machine learning tasks. 

High error rate and hardware limitations restricts constructions of classifiers 
in the space with dimension equal to the number of classes. 

For classifications with large number of classes following approach is applied. 
Choose a vector space~$E$ (we call it \textit{embedding space}) of large enough dimension.
Let a neural network inference to be a vector in~$E$ (\textit{embedding}). 
Then classes correspond to some clusters of points in~$E$. 
For a given dataset corresponding to a number of clusters in~$E$, 
classification problem is downshifted to the problem of the nearest cluster center 
search with some distance function. 
Choice of distance function depends on network construction and training. 

Cluster center locations in this approach are unknown and can change during network training. 
The goal of training is to put embeddings within the same cluster closer 
and to move clusters far from each other. 
In~\cite{NG1} it was proposed to predefine cluster centers in space of small dimension 
for classifications with small number of classes. 

The problem of classification with large number of classes becomes more 
challenging if we require that the set of classes to be extendable without 
re-training the classifier from scratch. 
This means that if we add some new classes, we only need to tune the network 
for new classes, but it still can classify old classes. 

For large number of classes there is a bottleneck in finding nearest 
cluster center or nearest embeddings. 
There are some approaches to circumvent it like 
vector databases~(\cite{RAG}) or graph search algorithms~(\cite{GS}). 

Problem of nearest cluster center search may be effectively solved if 
we train a network with predefined cluster centers even if we do not know 
the number of clusters in advance. 
The search can become effective if we choose some symmetric regular 
geometrical configuration of cluster centers. 
So, the problem is how to construct series of points such that we can 
extend it, embed in higher dimension and satisfy some additional conditions 
(for example, to choose points in some set, say, sphere if embeddings are 
normalized in network inference).

Suppose we have to find a number of nearest embeddings from dataset to a given embedding. 
Then we can optimize it by restriction of search on the subset of 
data samples corresponding only to the nearest clusters. 
At the same time, we do not need to run over all the cluster centers to find nearest ones. 

Although all the computations are done with floating point numbers, 
we here omit the possible inaccuracies, which are in the scope of separate investigation 
and consider all the data to be real numbers. 
In practice, we never need too large number of clusters that embeddings 
constructed in \texttt{float32} become close enough that 
rounding errors could spoil the results. 
Instead, one can increase the dimension. 

Let us introduce some notation. 

Let $E$ be a vector space of dimension~$n$ and $K\subset E$ be a compact subset of~$E$. 

\begin{definition}
  We call a sequence of point $\{C_{d,k}\}_{k=1}^{\infty} \subset K$ 
  a \textit{quasi-uniform scattering} with support~$K\subset \mathbb{R}^d$ and exponent~$R$,
  if for any~$n\in\mathbb{N}$ for the distance 
    \[
        D_{n,N} = \min_{m \ne n \atop m, n \leqslant N} \rho(C_{d,m}, C_{d,n})
    \]
    from point in the initial segment to closest neighbors in this segment 
    the following inequation holds:
    \[
        \frac{
            \max\limits_{1\leqslant n \leqslant N} D_{n,N}
        }{
            \min\limits_{1\leqslant n \leqslant N} D_{n,N}
        } \leqslant R.
    \]
\end{definition}

\begin{definition}
  \label{def:scattering}
  Let $K_d\subset \mathbb{R}^d$, $d\geqslant d_0$ be a sequence of compact sets.
  We call a set $\{\}_{n\in\mathbb{N},d\geqslant d_0}$ a 
  \textit{series of quasi-uniform scatterings} with exponent~$R$ and support $K_\bullet$
  if for any $d\geqslant d_0$ there exist maps $\iota_{d_1,d_2}\colon\mathbb{R}^{d_1} \to \mathbb{R}^{d_2}$
  such that for any $d_1<d_2$
    \begin{enumerate}
      \item $\iota_{d_2, d_3} \circ \iota_{d_1,d_2} = \iota_{d_1,d_3}$ (composition),
      \item $\iota_{d_2,d_2}(K_1) \subseteq K_2$ (embedding),
        \item $\rho(x,y) = \rho(\iota_{d_1,d_2}(x), \iota_{d_1,d_2})$ for 
          any~$x,y\in K_1$ (isometry),
        \item the sequence~$\{\iota_{d_1,d_2}(C_{d_1,n})\}_{n=1}^\infty$
            is a subsequence in~$\{C_{d_2,n}\}_{n=1}^\infty$ for 
            any~$d_1<d_2$.
    \end{enumerate}
\end{definition}

\begin{definition}
  \label{def:search}
  A sequence~$\{C_{d,n}\}_{n=1}^\infty$ admits the property of 
  \textit{fast search} if for any~$e\in\mathbb{R}^d$ and any~$n\in\mathbb{N}$
  there is an algorithm for search of nearest~$k$ points of an initial segment~$\{C_{d,n}\}_{n=1}^N$
  taking at most~$O(k)$ operations. 

  We say that the series of scatterings admit this property if all scattering 
  admits this property. 
\end{definition}

\begin{remark}
  Stated above estimation of algorithms like~$O(k)$ operations is not uniform in $d$. 
  So, the algorithms are linear in~$k$ and may have increasing asymptotic in~$d$. 
\end{remark}

We will consider here the cases where the sets~$K_d$ are spheres or surfaces of polyhedra. 

The main goal of present paper is to derive the construction of series of scatterings admitting fast search.
So, constructed series of sequences of points can be used as predefined centers of clusters.
Therefore, these constructions are useful for classifications without 
a classifying layer with fast search of nearest cluster centers and other embeddings. 
Also, this construction allows to extend the set of classes in such a way that 
a classifying neural network can classify already present classes and only needs fine tuning for new classes.

The paper is organized as follows. 
Section~2 contains some basic notation. 
Section~3 contains the construction of sets of points on the boundary of weight polyhedra. 
Section~4 contains nearest cluster center search algorithms. 

\section{Preliminaries}

We will use the following notation:
\begin{itemize}
    \item $G$~--- reductive group of some root system, say, $A_n$,
    \item $G \supset B \supset T$~--- some Borel subgroup and maximal torus,
    \item $\mathfrak{X}$~--- weight lattice with coordinates $x_1,\ldots,x_{n+1}$ 
        of~$G\supset B \supset T$,
    \item $W$~--- Weyl group,
    \item Weyl camera is a cone of dominant weights 
        (for $A_n$, $x_1\geqslant x_2\geqslant\ldots\geqslant x_{n+1}\geqslant 0$), 
        and its images for the $W$ actions
    \item $\lambda$~--- Young diagram with rows~$(\lambda_1,\ldots,\lambda_n)$,
        where~$n = \operatorname{rank}(\mathfrak{X})$,
    \item $\lambda'$~--- transposed Young diagram,
    \item $\mathfrak{X}_{\mathbb{Q}} := \mathfrak{X}\otimes_{\mathbb{Z}} \mathbb{Q}$~---
        vector space over~$\mathbb{Q}$ spanned by~$\mathfrak{X}$;
        analogously for~$\mathfrak{X}_\mathbb{R}$,
    \item word corresponding to a diagram is a word over alphabet~$\{1,\ldots,n+1\}$
        from rows from above to below,
    \item a Young Tableau is a Young diagram filled with~$\{1,\ldots,n+1\}$
        each element is in one of the boxes, such that entries in rows do decrease
        and entries in column increase,
    \item lexicographic ordering in words in alphabet~$\{x_1,\ldots,x_{n+1}\}$
        supposes that $x_1 \geqslant \ldots \geqslant x_{n+1}$
        (these are not geometrical inequalities, but symbol ordering: the symbol~$x_1$
        has higher order than~$x_2$ and so on).
\end{itemize}

One can find all the definitions for this notation in~\cite{R}. 

Cosine distance between a pair of vectors~$u,v$ is defined as
\[
  1 - \frac{u \cdot v}{||u||_2||v||_2}.
\]
It is not a distance function on~$\mathbb{R}^n$, but on the unit sphere~$S^{n-1}\subset\mathbb{R}^n$ 
it is a well defined metric.

\section{Weight polyhedron and its boundary points construction algorithm}

We need to construct a series of points to predefine them as 
centers of clusters in the embedding space. 
We choose them as integral points (i.\,e. lying in the lattice~$\mathfrak{X}$)
on the boundary of some symmetric polyhedron. 
Given a dominant weight~$w_\lambda$, we choose a polyhedron~$P_\lambda$ 
as the convex hull of Weyl group~$W$ orbit of~$w_\lambda$. 
For the case of root system~$A_n$, this orbit is just a set of 
coordinates permutations of vector~$w_\lambda$.

We start the construction from a Young diagram~$\lambda$, where 
vectors in~$W\cdot w_\lambda$ correspond to semi-standard Young tableau. 
Although there can be many semi-standard Young tableau with the same word,
they give the same point in the space~$\mathfrak{X}_{\mathbb{R}}$, if 
the words can be obtained by a single permutation from each other. 
So, there is a canonical way to choose a single word representing 
a weight. We can just take non-decreasing words.

Although one could take a set of points to represent a set of cluster centers 
and train a neural network with euclidean distance, the cosine distance 
is more widely used and usually provides better convergence. 
It does not distinguish point on the same ray from zero. 
So, we take only the points on the boundary of~$P_\lambda$ to avoid
getting indistinguishable or very close in cosine distance points. 

For cosine distance, it does not matter if the set of point is located 
on a unit sphere, or on a boundary of a convex compact set containing zero inside. 
So, it is easier to construct the points on a polyhedron boundary. 
At the same time, search algorithms of the closest cluster center 
slightly differ for euclidean and cosine distance~(see~\S4).

In this section we cover in details two questions:
\begin{itemize}
    \item
        Given a dominant weight $w_\lambda$ of a reductive group~$G$, how to
        effectively construct all weights~$\partial P_\lambda \cap \mathfrak{X}$ 
        on the weight polyhedron~$P_\lambda$ boundary?
    \item How to proceed the subdivision step and construct all the points in
        $\partial P_\lambda \cap (\mathfrak{X} / 2^k)$ for~$k\in\mathbb{N}$?
\end{itemize}

There is an important property of these sets of points: if we count Euclidean 
distance from each point of~$\partial P_\lambda \cap (\mathfrak{X} / 2^{k})$ for
$k\in\mathbb{Z}_{\geqslant 0}$ for the nearest neighbor, all these distances will coincide.

Although the same statement fails for cosine distance counted from the 
center~$c_\lambda$ of $P_\lambda$, the ratio of maximal and minimal of them cannot exceed
$$
    \frac{
        \max\limits_{w\in\partial P_\lambda}||w-c_\lambda||
    }{
        \min\limits_{w\in\partial P_\lambda}||w-c_\lambda||
    }.
$$

\subsection{Construction of $\partial P_{\lambda} \cap \mathfrak{X}$}

First, we need to construct algorithm  enumerating all semi-standard Young 
tableau of form~$\lambda$ corresponding to a non-decreasing word and a point 
inside the dominant Weyl camera~$C$.

We start with the initial Young tableau obtained from $\lambda$ by
filling the first row by~$1$, the second by~$2$ and so on.

After, we run over Young tableau in lexicographic ordering and check the rules below.
If a diagram does not satisfy the rules, we skip all Young tableaux lower it in lexicographic order.

\begin{itemize}
  \item[\textbf{Rule 1.}] Condition on Weyl camera: number of <<$1$>> entries
    is not less than the number of <<$2$>> entries, the number of <<$3$>> entries does 
    not exceed the number of $2$ entries and so on.
  \item[\textbf{Rule 2.}] Boundary condition.
    Suppose that Young tableau is filled in such way that the corresponding word
    is non-decreasing. Then there is at least one row~$\lambda_i$ finishing
    with the box containing~$i$.
\end{itemize}

\begin{remark}
    There is a vertex $w_\lambda = (\lambda_1,\ldots,\lambda_n)\in\mathfrak{X}$ of~$P_\lambda$.
    All the other vertices of~$P_\lambda$ are just $W$-orbit of~$w_\lambda$.
    
    Faces of~$P_\lambda$ are given by inequalities
    \begin{equation}
      \label{eq:ehw}
      \begin{aligned}
          x_1 &\leqslant \lambda_1, \\
          x_1+x_2 &\leqslant \lambda_1 + \lambda_2, \\
          &\ldots \
      \end{aligned}
    \end{equation}
    So, to get a point on~$\partial P_\lambda$, we should check that
    at least one of inequations is not strict (i.\,e.~becomes an equation).
\end{remark}

\begin{remark}
    Algorithm takes $O(k)$ computational operations to get $k$ points.
    For each move to another point, one should perform constant number of operations~($O(1)$).
    If we get a new boundary point, it contributes to both number of found points
    and the number of computational operations.
    If a new point~$w'$ does not satisfy some of the rules, 
    the point is in the unit cube with upper corner at~$w$ (previous point).
    So, for $k$ resulting points we verify at most~$2^n\cdot k$ points, so
    we make $O(k)$ computational operations.
\end{remark}

\begin{example}
  Consider the group $\operatorname{GL}(4)$ corresponding to $A_3$ 
  root system and $\lambda = \Yvcentermath1\yng(2,1,1)$.
  Actually, $\Yvcentermath1\young(11,2,3)$ is the only correct filled diagram. 
  It contains two $1$'s, one $2$ and $3$ and no $4$, so we get point $(2,1,1,0)$. 
  Other points are obtained as permutations. 
  \medskip
  We can get the same polyhedron, if we subtract a constant from each coordinate 
  and permute~$(1,0,0,-1)$.
  Actually, for $A_n$ and diagram $\lambda=(2,1\ldots,1,0)$ we always get here
  the same case of unique point and its permutations.
\end{example}

\begin{example}
  Consider the group $\operatorname{GL}(5)$ corresponding to $A_4$ 
  root system and $\lambda = \Yvcentermath1\yng(2,2,1,1)$.
  There are only two ways to fill the diagram:
  \[
    \Yvcentermath1
    \young(11,22,3,4) \text{ and } \young(11,23,4,5).
  \]
  So, we get points $(2,2,1,1,0)$ and $(2,1,1,1,1)$ and their permutations.
  Or, equivalently, $(1,1,0,0,-1)$ and $(1,0,0,0,0)$. 
  \medskip

  The same holds for $\lambda=(2,2,1,\ldots,1,0)$ for $A_n$.
\end{example}

\begin{example}
  Consider the group $\operatorname{GL}(4)$ corresponding to $A_3$ 
  root system and $\lambda = \Yvcentermath1\yng(3,2,1)$.
  Filled diagrams are
  \[
    \Yvcentermath1
    \young(111,22,3) \text{, } \young(111,23,4) \text{ and } \young(112,23,3).
  \]
  The points are~$(3,2,1,0)$, $(3,1,1,1)$, $(2,2,2,0)$.
  These points and there permutations are integral boundary points of the permutahedron.
  There is a diagram
  \[
    \Yvcentermath1
    \young(112,23,4)
  \]
  corresponding to the point~$(2,2,1,1)$.
  It does not satisfy the boundary rule and corresponds to a 
  point inside the polyhedron
  (Although it lies in the dominant cone). 
  So, we do not consider it.
\end{example}

\begin{theorem}
  \label{theo:CPX}
    The algorithm above runs over all points of $C \cap \partial  P_\lambda \cap \mathfrak{X}$.
\end{theorem}

\begin{proof}
  Enumeration of all non-decreasing words in lexicographic order effectively 
  runs over all points in~$\mathfrak{X}\cap P_\lambda$.
  We deduce enumeration of all words to only dominant and then act by~$W$ and get all 
  points, so we visit all the points. 
  Next, we consider only points satisfying boundary conditions for faces 
  containing~$w_\lambda$. 
  Since we have already deduced enumeration only to the dominant cone,
  we can consider only boundary conditions arising from faces containing~$w_\lambda$. 
\end{proof}

%We do not consider points outside~$C$, because we will obtain them as~$W$-orbits
%of computed points corresponding to dominant weights.

\subsection{Construction of $\partial P_{\lambda} \cap (\mathfrak{X} / 2^k)$}

We will run over points (or, semi-standard Young tableau) collected on the
previous step of construction (or, step for~$k>1$).

Firstly, we multiply the diagram horizontally be~$2$, i.\,e.take each column two times
(for $k=1$, or $2^{k-m}$ times for $k>1$ and a diagram from step~$m$, where~$m=0$
corresponds to initial polyhedron construction). Denote obtained Young tableau by~$\mu$.

In lexicographic order, we run over subsets of rows of~$\mu$:
\begin{itemize}
    \item for a subset, we increment the last box entry and mark it;
    \item if~$k>1$, we increment last unmarked box;
    \item when we multiply tableau horizontally, marks are populated also;
    \item check the conditions of Weyl camera, boundary condition,
        increase of word corresponding to a diagram;
        if a diagram fail any condition, we skip it.
\end{itemize}

\begin{example}
  Consider the group~$\operatorname{GL}(4)$ corresponding 
  to~$A_3$ and~$\lambda=\Yvcentermath1\yng(2,1,1)$.
  \medskip 

  The only semi-standard Young tableau corresponding to a 
  dominant weight on the boundary of~$P_\lambda$ is~$\Yvcentermath1\young(11,2,3)$.
  Let us proceed two subdivision steps. 
  First, multiply the diagram: $\Yvcentermath1\young(1111,22,33)$.
  Now, add marks and increment some cells:
  \[
    \Yvcentermath1
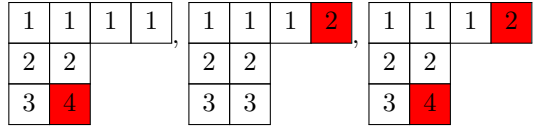
    \begin{ytableau}
      1 & 1 & 1 & 1 \\ 
      2 & 2 \\ 
      3 & *(red) 4 \\
    \end{ytableau},\,
    \begin{ytableau}
      1 & 1 & 1 & *(red) 2 \\ 
      2 & 2 \\ 
      3 & 3 \\
    \end{ytableau},\,
    \begin{ytableau}
      1 & 1 & 1 & *(red) 2 \\ 
      2 & 2 \\ 
      3 & *(red)4 \\
    \end{ytableau}.
  \]
  Initial point gives point~$(2,1,1,0)$. These diagrams 
  give~$\left(2,1,\frac12,\frac12\right)$, $\left(\frac32,\frac32,1,0\right)$ 
  and~$\left(\frac32,\frac32,\frac12,\frac12\right)$.
  \medskip

  Let us make another subdivision step and get 
  \[
    \Yvcentermath1
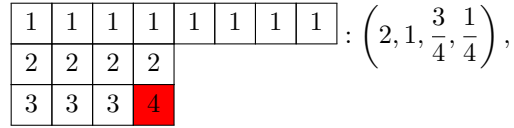
    \begin{ytableau}
      1 & 1 & 1 & 1 & 1 & 1 & 1 & 1 \\ 
      2 & 2 & 2 & 2 \\ 
      3 & 3 & 3 & *(red) 4 \\
    \end{ytableau}: \left(2,1,\frac34,\frac14\right),\,
  \]
  \[
    \Yvcentermath1
    
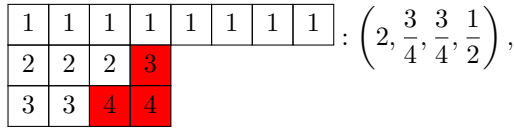
\begin{ytableau}
      1 & 1 & 1 & 1 & 1 & 1 & 1 & 1 \\ 
      2 & 2 & 2 & *(red)3 \\ 
      3 & 3 & *(red)4 & *(red) 4 \\
    \end{ytableau}: \left(2,\frac34,\frac34,\frac12\right),\,
  \]
  \[
    \Yvcentermath1
    
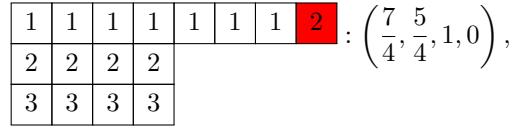
\begin{ytableau}
      1 & 1 & 1 & 1 & 1 & 1 & 1 & *(red) 2 \\ 
      2 & 2 & 2 & 2 \\ 
      3 & 3 & 3 & 3 \\
    \end{ytableau}: \left(\frac74,\frac54,1,0\right),\,
  \]
  \[
    \Yvcentermath1
    
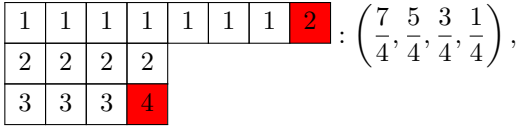
\begin{ytableau}
      1 & 1 & 1 & 1 & 1 & 1 & 1 & *(red) 2 \\ 
      2 & 2 & 2 & 2 \\ 
      3 & 3 & 3 & *(red)4 \\
    \end{ytableau}: \left(\frac74,\frac54,\frac34,\frac14\right),\,
  \]
  \[
    \Yvcentermath1
    
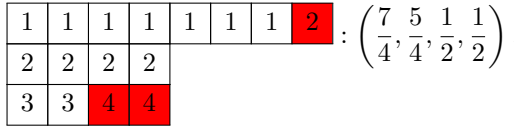
\begin{ytableau}
      1 & 1 & 1 & 1 & 1 & 1 & 1 & *(red) 2 \\ 
      2 & 2 & 2 & 2 \\ 
      3 & 3 & *(red)4 & *(red)4 \\
    \end{ytableau}: \left(\frac74,\frac54,\frac12,\frac12\right).
  \]
\end{example}

\begin{example}
  Consider the group~$\operatorname{GL}(5)$ corresponding 
  to~$A_4$ and~$\lambda=\Yvcentermath1\yng(2,2,1,1)$.
  \medskip 

  Initial construction gives two diagrams corresponding to two points before $W$ action:
  \[
    \Yvcentermath1
    \young(11,22,3,4): (2,2,1,1,0) \text{ and } \young(11,23,4,5): (2,1,1,1,1).
  \]
  Let us make subdivision and get more points:
  \[
    \Yvcentermath1
    \begin{ytableau}
      1 & 1 & 1 & 1 \\ 
      2 & 2 & 2 & 2 \\ 
      3 & 3 \\
      4 & *(red)5 \\
    \end{ytableau}: \left(2,2,1,\frac12,\frac12\right),\quad
  %\]
  %\[
    \Yvcentermath1
    \begin{ytableau}
      1 & 1 & 1 & 1 \\ 
      2 & 2 & 2 & *(red)3 \\ 
      3 & 3 \\
      4 & 4 \\
    \end{ytableau}: \left(2,\frac32,\frac32,1,0\right),\,
  \]
  \[
    \Yvcentermath1
    \begin{ytableau}
      1 & 1 & 1 & 1 \\ 
      2 & 2 & 2 & *(red)3 \\ 
      3 & 3 \\
      4 & *(red)5 \\
    \end{ytableau}: \left(2,\frac32,\frac32,\frac12,\frac12\right),\quad
  %\]
  %\[
    \Yvcentermath1
    \begin{ytableau}
      1 & 1 & 1 & 1 \\ 
      2 & 2 & 2 & *(red)3 \\ 
      3 & *(red)4 \\
      4 & *(red)5 \\
    \end{ytableau}: \left(2,\frac32,1,1,\frac12\right),\,
  \]
  \[
    \Yvcentermath1
    \begin{ytableau}
      1 & 1 & 1 & *(red)2 \\ 
      2 & 2 & 3 & 3 \\ 
      4 & 4 \\
      5 & 5 \\
    \end{ytableau}: \left(\frac32,\frac32,1,1,1\right).
  \]
\end{example}

\begin{theorem}
  \label{theo:CPX2}
  The algorithm above runs over all points of $C \cap \partial  P_\lambda \cap (\mathfrak{X} / 2^k)$ for any~$k\geqslant1$.
    %The algorithm above gets $k$ points for~$O(k)$.
\end{theorem}

\begin{proof}
  Horizontally doubled Young diagram corresponds to run over polyhedron 
  points $\partial P_{2\lambda}\cap\mathfrak{X}=\partial (2P_{\lambda})\cap\mathfrak{X}$ 
  or, equivvalently, $\partial P_\lambda\cap (\mathfrak{X}/2)$. 
  Points with all even coordinates can be obtained from taken on 
  the previous subdivision step, if we double semi-standard Young tableau with all entries.
  Remaining part of proof is similar to Theorem~\ref{theo:CPX}.
\end{proof}

When the step of subdivision is done, all the euclidean distances from a point
to the nearest neighbor coincide. Between the steps the ratio of highest
and lowest ones is equal or does not exceed~$2$.
This shows that the constructed sequence of points satisfies conditions of~\eqref{def:scattering}.
So, we can conclude:

\begin{proposition}
  Given a Young diagram~$\lambda$, consider the sequence~$C$ of points in 
  the compact set~$K := P_\lambda \cap \mathfrak{X}_{\mathbb{R}}$ in the following way.
  For each~$k\geqslant1$ construct the points in $C \cap \partial P_\lambda \cap \mathfrak{X}_{\mathbb{Q}}$ as above,
  then take all their $W$-orbits, sort the set of new points lexicographically and queue to the sequence~$C$.
  Then sequence~$C$ is a quasi-uniform scattering with exponent~$R=2$.
\end{proposition}

\begin{example}
  Consider $A_3$ and~$\Yvcentermath1\lambda=\yng(2,1,1)$ the sequence~$C$ as follows:
  \begin{multline*}
    (2,1,1,0), (2,1,0,1), (2,0,1,1), 
    (1,2,1,0), (1,2,0,1), (1,1,2,0), \\ (1,1,0,2),
    (1,0,2,1), (1,0,1,2),
    (0,2,1,1), (0,1,2,1), (0,1,1,2), \\
    \left(2,1,\frac12,\frac12\right)
    \left(2,\frac12,1,\frac12\right)
    \left(2,\frac12,\frac12,1\right)
    \left(\frac32,\frac32,1,0\right)
    \left(\frac32,\frac32,\frac12,\frac12\right), ...
  \end{multline*}
  Denote sequence for $A_4$ and $(2,1,1,1)$ by~$C'$. 
  If we append~$1$ at the end of each~$C$ entry, it becomes a subsequence in~$C'$. 
  It is the way how scatterings can be chained into a series of quasi-uniform scatterings. 
  \medskip

  We should append~$1$, because we extend groups~$\operatorname{GL}(4)\subset\operatorname{GL}(5)$ and extend diagrams~$(2,1,1)\mapsto(2,1,1,1)$.
  So, highest weight vector is obtained from diagram with an additional box containing a new digit not used before, this gives~$1$ in a new coordinate.
  \medskip

  It is not a unique way how such scatterings can be chained into series. 
  For instance, we can chain~$(A_n, \lambda=(2,2,1,\ldots,1))$ and get other series of points. 
  Or, for~$(A_n,\lambda=(n,n-1,\ldots,1))$ we get series of points on permutahedrons.
\end{example}

\begin{remark}
  Theorem~\ref{theo:CPX2} states the contraction of quasi-uniform scatterings in euclidean metric. 
  If one consider the case of cosine metric, the statement holds for some bigger~$R$, 
  namely,~$R=2\frac{\max\limits_{v\in P_{\lambda}}||v||_2}{\min\limits_{v\in P_{\lambda}}||v||_2}$.
\end{remark}

\section{Search algorithm}

Given some number of predefined cluster centers and data samples, 
for an arbitrary embedding~$e$ we are interested to solve two problems:
\begin{enumerate}
  \item find nearest cluster center,
  \item find nearest one or more data embeddings. 
\end{enumerate}
We can predefine cluster centers in order to simplify search algorithm. 
Although data samples would not be any kind of regular in the embedding space,
we can decrease search range for data samples and consider only data 
corresponding to some number of close clusters.
For a huge number of classes and locally stored clusters this can 
significantly simplify the problem of nearest data embedding search. 

For both cases, the problem is deduced to search of one or more nearest 
cluster centers for an arbitrary point in the embedding space.

Suppose we are given an embedding $e=(e_1,\ldots,e_n) \in \mathfrak{X}_{\mathbb{R}}$.
We have to find the closest of constructed points for~$O(1)$ operations (or, $k$ closest points for~$O(k)$ operations).

Without loss of generality, we can assume that~$e$ lies in the dominant cone 
(for $A_n$ it means that $e_1 \geqslant \ldots \geqslant e_n$).
Otherwise we can act by~$W$ and take one orbit point~$w\cdot e$, $w \in W$ in the dominant cone. 
After we find nearest cluster centers $c_1,\ldots,c_k$ for~$w\cdot e$, cluster centers 
$w^{-1}\cdot c_i$, $i=1,\ldots,k$ will be nearest cluster centers for~$e$.  

If a point~$e$ is strictly inside a dominant cone, the closest faces to $e$ are
faces containing~$w_\lambda$. Otherwise, there will be other closest faces.
Anyway, there is a finite fixed set of hyper-faces of $P_\lambda$ nearest to~$e$.

Actually, we can simply find the unique nearest cluster center. 
If we need $k>1$ centers and we have nearest one~$c$, next~$k-1$ can be found by 
width search over~$c$ neighborhood in hyper-faces containing~$c$. 
Since~$c$ lies in at most~$d$ hyper-faces at the same time (this case corresponds to a vertex of~$P_\lambda$),
if we can consider a linearly in~$d$ bounded number of cluster centers to find each next one. 
Therefore, we can find~$k$ nearest cluster centers in $O(k)$ computational operations. 
It only remains to find the nearest cluster center with a finite number of operations. 

Search of the closest cluster center is divided into two parts:
\begin{enumerate}
  \item Find the closest face~$F$ of $P_\lambda$ (not necessarily hyper-face) such that 
    the orthogonal projection~$e'$ of~$e$ into the affine space spanned by $F$ is contained in 
    the interior of~$F$ or $F$ is a $0$-dimensional face, i.\,e.~a~vertex. 
  \item Constructed cluster centers inside~$F$ are part of lattice points 
    inside the affine space. 
    So, the nearest ones can be found just by rounding coordinates of~$e'$ 
    to nearest point with integer numbers (or, integral coordinates 
    divided by~$2^s$, if we have made~$s$ subdivision operations).
\end{enumerate}

The second part is quite simple, so it only remains 
to find the face~$F$ and the projection~$e'$. 

\subsection{Search for the nearest point in a face~$F$ of~$P_\lambda$}

We can find the face~$F$ iteratively. 
The closest to~$e$ hyper-face is one of the faces containing~$w_{\lambda}$. 
Containing them affine subspaces are given by linear equations~\eqref{eq:ehw}.
One can normalize equations and find distance from~$e$ to these affine spaces.
Then one can project~$e$ to a point~$e_1$ on the closest affine space. 
If $e_1$ is in the interior of hyper-face, then~$e'=e_1$, and~$F$ is found. 
Otherwise we can find the next closest hyper-face and project~$e_1$ into the next point~$e_2$ 
inside an affine space of codimension~$2$ and so on. 

This simple reasoning shows that the closest needed face can be 
found in finite number of operations.
Maybe it can be optimized using ordering the distances by absolute value,
then replacement th numbers with orders with same signs and then 
by search of result in hash-table, if for same key the result is already inside the hash-table. 

Another way to find for each face~$F$ adjacent to~$w_{\lambda}$ is to precompute 
areas of dominant cone from where orthogonal projection to the affine 
space spanned by~$F$ is strictly  inside~$F$.
These areas are polyhedral cones. 
Therefore, it is enough for each point~$e$ to calculate a series of linear functions. 
Each face~$F$ corresponds to a subset of these functions, and must be non-negative on the point. 
This algorithm is more compicated, but it can be simpler vectorized and applied on GPU.
These linear functioins can be combined into matrix, so there computation 
on a batch of points~$e$ is just a multiplication of two matrices. 
Remaining part can be implemented as decision tree and parallelized. 

This finishes the proof that constructed above scatterings admits the fast search condition~\ref{def:search}.

\section{Conclusion}

Although the provided method maximally simplifies the nearest embeddings search problem,
it makes network training more complicated.

As in any complicated computational problem,
here is no a silver bullet.
The algorithm does not give a good solution if one takes arbitrary~$n$ for root system $A_n$
and takes an arbitrary dominant weight~$\lambda$.
Here the following problems arise:
\begin{itemize}
    \item High dimension always brings many problems with slow convergence.
    \item In some low dimensions many points become too close to nearest neighbors.
        So, network training can diverge.
\end{itemize}

So, dimension~$n$ and initial highest weight~$w_\lambda$ should be chosen carefully.
But if one knows how many points are needed, then one can precompute a dimension such that 
taken points will be far enough from each other.

So, the hard problem of bottleneck at nearest points search is replaced with
precomputation of parameters of the points configuration.
Also, network training become more complicated.
For some practical cases (see~\cite{NG2}) this algorithm gives good working examples of training.
In a general case, it gives a number of geometrical constructions and ways to tune it for
a particular situation.

For instance, almost anywhere border cases give bad results.
If one take the whole root system as a set of points, the dimension would be high.
If one take $\pi_1+\ldots+\pi_n$ for $A_n$, the procedure gives a permutahedron
in very small dimension, but the points will be very close.
Some intermediate weights can provide better results.

Although all given examples and tested cases only deal with $A$ type root systems,
other classical types $B$, $C$, $D$ or exeptional root systems can also 
give potentially interesting scatterings with useful properties. 

\section{Acknowledgmenets}

The author is grateful to his Kryptonite colleagues Vasily Dolmatov, Dr. Nikita Gabdullin,
Dr. Anton Raskovalov and Ilya Androsov for fruitful discussions of topic and results 
and examination of the construction with real data and big neural network models,
where the approach turned out to be applicable and useful.

\end{document}